\documentclass[%
reprint%
,onecolumn%
,jmp%
,a4paper,superscriptaddress,nofootinbib,showkeys]{revtex4-2}

\usepackage{mathrsfs}

\usepackage{amsmath,amssymb,bm,cancel}
\usepackage{amsthm}
\usepackage{mathtools}

\newtheorem{theorem}{Theorem}[section]
\newtheorem{proposition}{Proposition}[section]
\newtheorem{lemma}{Lemma}[section]
\newtheorem{corollary}{Corollary}[section]

\theoremstyle{definition}
\newtheorem{definition}{Definition}[section]

\theoremstyle{remark}
\newtheorem{remark}{Remark}[section]
\newtheorem{example}{Example}[section]

\usepackage{graphicx}
\usepackage[dvipsnames]{xcolor}

\newcommand{\dble}{\mathbb{D}(M)}
\DeclareMathOperator{\interior}{int}

\begin{document}

% Paragraphs:

\setlength\parindent{0pt}
\setlength{\parskip}{0.5em}

\title{Integrability of normal distributions \\
Part 2: Neat foliations by manifolds with boundary}

\author{David Perrella}
\author{David Pfefferl{\'e}}
\author{Luchezar Stoyanov}
\affiliation{The University of Western Australia, 35 Stirling Highway, Crawley WA 6009, Australia}

\maketitle

\section{Abstract}
This paper completes the foundations of neatly integrable normal distribution theory on manifolds with boundary. Normal distributions are those which contain vectors transverse to the boundary along its entirety. The theory is observed to be entirely analogous with the theory of integrable distributions on manifolds due to Stefan and Sussmann. The main result is a one-to-one correspondence between so-called neatly integrable normal distributions and neat foliations by manifolds with boundary. Neat foliations are allowed to have non-constant dimension and the leaves have boundary contained in the ambient boundary. The leaves satisfy a characteristic property formally identical to that of weakly embedded submanifolds except in the category of manifolds with boundary.

\section{Introduction}
This work harvests the global fruits of the integrability theory for normal distributions on manifolds with boundary \cite{Part1}. These are distributions which contain vectors transverse to the boundary along its entirety. It is seen that the foundations of normal distributions run parallel with the established theory in the case without boundary under our chosen notions of ``integral", ``regularity", ``maximality" and ``foliation".

In particular, our ``integrals" are so-called \emph{integral manifolds with boundary} and precisely when the boundary is contained by the ambient boundary, we will see that these so-called \emph{neat} integrals are ``weakly embedded". The characteristic property satisfied is identical to that of ordinary weakly embedded submanifolds, except the category is manifolds with boundary. 

Our notion of maximality takes account of all integral manifolds with boundary, and the neat integral manifolds are seen to be the largest. The existence of maximal integral manifolds with boundary guarantees existence of certain foliations and may be decided with the same ease as in the case without boundary. The foliations guaranteed are so-called \emph{neat foliations by manifolds with boundary}. These foliations may have non-constant dimension. The theory in this generality appears to be absent from the literature.

In the case of constant rank, these foliations are also known as \emph{foliations transverse to the boundary} and are used to define the notion of \emph{foliated cobordism} \cite{mizutani1983cobordism,tsuboi1992godbillon,farrell1986h,reinhart1964cobordism,fukui1978remark}. As mentioned in \cite{Part1}, these foliations were considered by Noakes \cite{Noakes} where he generalised Bott's Theorem \cite{Bott} as well as demonstrate how these foliations may be used to guarantee (ordinary) foliations.

This paper is structured as follows. We first review integral manifolds of distributions as well as foliations in the case of manifolds (without boundary) in Section \ref{sec:review_int_mflds_fol}. In Section \ref{sec:neat_main_res}, we discuss neat integral manifolds of normal distributions, neat foliations by manifolds with boundary and state our main results. We then give a proof of these results in Section \ref{sec:Proof}. After this, we make some additional remarks on the structure of neat integral manifolds with boundary in Section \ref{sec:remark}. In Appendix \ref{sec:appendix}, we discuss in detail the smooth properties of weakly embedded submanifolds with boundary.

\section{Review of integral manifolds and foliations}
\label{sec:review_int_mflds_fol}
The definitions of the standard terms we use from smooth manifold theory coincide with that of Lee's book \cite{Lee} and of Rudolf and Schmidt's book \cite{RudolphSchmidt}. Here we will review the known results in the case of distributions on manifolds (without boundary) which we will replicate for normal distributions. We first recall the relevant definitions following Rudolf and Schmidt \cite{RudolphSchmidt} and \cite{Part1}. For the following, let $M$ be a manifold without boundary.

For an arbitrary subset $D \subset TM$, a \textbf{\emph{$\bm{D}$-section}} is a vector field $X$ on $M$ such that for all $p \in M$, $X|_p \in D$. We also denote $D_p = D \cap T_pM$ for $p \in M$. With this, $D$ is called a \textbf{\emph{distribution on $\bm{M}$}} if for all $p \in M$, $D_p$ is a vector subspace of $T_pM$ and for any $v \in D_p$, there exists a $D$-section $X$ with $X|_p = v$. The \textbf{\emph{rank of $\bm{D}$}} is the map $r : M \to \mathbb{N}_0$ given by $r(p) = \dim D_p$ (where $\mathbb{N}_0$ denotes the non-negative integers). For a smooth map $f : S \to M$ between manifolds, we set $f^*D = (Tf)^{-1}(D)$ and call $f^*D$ the \textbf{\emph{pullback of $\bm{D}$ on to $\bm{S}$}} (where $Tf:TS\to TM$ denotes the tangent map).

Let $D$ be a distribution on $M$ of rank $r$. An \textbf{\emph{integral manifold of $\bm{D}$}} is a connected immersed submanifold $I$ of $M$ with inclusion $i : I \subset M$ such that, for $p \in I$, $Ti|_p(T_pI) = D_p$. We say that $D$ is \textbf{\emph{integrable}} if through every point of $M$, there passes an integral manifold of $D$.

One can detail the ``regularity" of an integral manifold of a distribution; they are in fact weakly embedded. Recall, an immersed submanifold $S$ of $M$ is said to be \textbf{\emph{weakly embedded}} if for any smooth map $f : N \to M$ between manifolds with $f(N) \subset S$, the induced map $N \to S$ by $f$ is also smooth. This regularity is the contents of \cite[Proposition 3.5.15]{RudolphSchmidt} as follows.

\begin{proposition}\label{boundarylessweakembedded}
Let $D$ an integrable distribution on $M$. Then, all integral manifolds of $D$ are weakly embedded in $M$.
\end{proposition}

Another important feature of integrable distributions is the existence of maximal integral manifolds. Recall (e.g. \cite{Part1}), an integral manifold $I$ of a distribution $D$ on a manifold $M$ is \textbf{\emph{maximal}} if, for any integral manifold $H$ of $D$ with a point in common with $I$, $H$ is an open submanifold of $I$. This feature is the contents of \cite[Theorem 3.5.17]{RudolphSchmidt} which we express using our terminology as follows.

\begin{theorem}\label{boundarylessmaximal}
Let $M$ be a manifold and $D$ an integrable distribution on $M$. Then, through every point of $M$, there passes a unique maximal integral manifold of $D$.
\end{theorem}

We now address foliations. Let $n = \dim M$. A \textbf{\emph{foliation}} $\mathcal{F}$ is a partition of $M$ by weakly embedded submanifolds such that for every $p \in M$ there exists a chart $(U,\varphi)$ about $p$ satisfying the following.

\begin{enumerate}
  \item $\varphi(p) = 0$ and $\varphi(U) = (-\epsilon,\epsilon)^n$ for some $\epsilon > 0$.
  \item All $I \in \mathcal{F}$ are invariant under the flows of $\partial_1,...,\partial_R$ where $R$ denotes the dimension of the unique $H \in \mathcal{F}$ in the partition containing $p$.
\end{enumerate}

Given a foliation $\mathcal{F}$, an element $I \in \mathcal{F}$ is called a \textbf{\emph{leaf of $\bm{\mathcal{F}}$}}. The function $r : M \to \mathbb{N}_0$ given by $r(p) = \dim I$ where $I$ is the unique leaf of $\mathcal{F}$ containing $p$ is called the \textbf{\emph{dimension of $\bm{\mathcal{F}}$}}.

The following is Proposition 3.5.21 in Rudolf and Schmidt's book \cite{RudolphSchmidt} which establishes a correspondence between integrable distributions and foliations.

\begin{proposition}\label{dist and foliation correspondence}
Let $M$ be a manifold. The assignment of the family of maximal integral manifolds to an integrable distribution defines a bijection between integrable distributions and foliations on $M$. Integrable distributions of constant rank $r$ thereby correspond to foliations of constant dimension $r$.
\end{proposition}

We will now make similar definitions in the case of normal distributions on manifolds with boundary and state the analogous results there.

\section{Neat integral manifolds, neat foliations, and the main results}
\label{sec:neat_main_res}

We first recall some definitions surrounding integral manifolds from \cite{Part1} as well as make some additional definitions which mirror the case without boundary. For the following, let $M$ be a manifold with boundary. The definition of distribution is unchanged as compared with \ref{sec:review_int_mflds_fol} except now $M$ has boundary.

Let $D$ be a distribution on $M$. A connected immersed submanifold $L$ with boundary of $M$ is called an \textbf{\emph{integral manifold with boundary of $\bm{D}$}} if for $p \in L$, $Tl|_p(T_pL) = D_p$ where $l : L \subset M$. We will say that $D$ is \textbf{\emph{integrable by boundaries}} if through every point of $M$, there passes an integral manifold with boundary of $D$.

An integral manifold with boundary $L$ of $D$ is said to be \textbf{\emph{neat}} if $\partial L \subset \partial M$. Accordingly, $D$ is said to be \textbf{\emph{neatly integrable by boundaries}} (or \textbf{\emph{neatly integrable}} for short) if through any point in $M$, there passes a neat integral manifold of $D$.

Our results are centered around normal distributions. Recall, the distribution $D$ is said to be \textbf{\emph{normal}} if for every $p \in \partial M$, there exists a vector in $D_p$ transverse to $\partial M$.

For a submanifold with boundary $S$ of $M$, we say that $S$ is \textbf{\emph{weakly embedded in $\bm{M}$}} if for any smooth map $f : N \to M$ between manifolds with boundary such that $f(N) \subset S$, the induced map $N \to S$ by $f$ is also smooth.

This details the level of ``regularity" for neat integral manifolds of neatly integrable normal distributions. The following is an analogue of Proposition \ref{boundarylessweakembedded}.

\begin{proposition}\label{normal dist weak embedd}
Let $M$ be a manifold with boundary and $D$ be a normal distribution on $M$ which is neatly integrable. Then, the neat integral manifolds are weakly embedded in $M$.
\end{proposition}

Concerning ``maximality", we must slightly weaken the definition compared to the case without boundary. The reason for this is illustrated by Example \ref{necessaryweaken}. For an integral manifold $L$ with boundary of $D$ on $M$, $L$ is said to be \textbf{\emph{maximal}} if, for any integral manifold $K$ with boundary of $D$ with a point in common with $L$, $K$ is a submanifold with boundary of $L$.

One of the main results of this paper is the analogue of Theorem \ref{boundarylessmaximal}. The statement is as follows.

\begin{theorem}\label{normal dist maximal}
Let $M$ be a manifold with boundary and $D$ be a normal distribution on $M$ which is neatly integrable. Then, for each point of $M$, there passes a unique maximal integral manifold with boundary. Moreover, the maximal integral manifolds with boundary are neat.
\end{theorem}

Our notion of foliation by manifolds with boundary must be combinatorially modified as compared to the notion of foliation. This is so that the leaves in charts have their boundaries along the boundary $\partial \mathbb{H}^n$ of the half-space $\mathbb{H}^n$. This necessity is similar to the modification of ``adapted chart" found in \cite{Part1}.

Set $n = \dim M$. A \textbf{\emph{neat foliation with boundaries}} $\mathcal{F}$ is a partition of $M$ by weakly embedded submanifolds with boundary such that $\partial L \subset \partial M$ for all $L \in \mathcal{F}$ and for every $p \in M$ there exists a chart $(U,\varphi)$ about $p$ satisfying the following (where $R$ denotes the dimension of the unique $K \in \mathcal{F}$ containing $p$).

\begin{enumerate}
  \item $\varphi(p) = 0$ and $\varphi(U) = (-\epsilon,\epsilon)^n$ for some $\epsilon > 0$ if $p \in \interior M$ and $\varphi(U) = (-\epsilon,\epsilon)^n \cap \mathbb{H}^n$ for some $\epsilon > 0$ otherwise.
  \item For all $L \in \mathcal{F}$, all coordinate vector fields $\partial_i$ with $n-R+1 \leq i \leq n$, all integral curves of $\partial_i$ intersecting $L$ are contained by $L$.
\end{enumerate}

Theorem \ref{normal dist maximal} along with the Stefan-Sussmann theorem for normal distributions in \cite{Part1} give an analogue of Proposition \ref{dist and foliation correspondence} on the correspondence between integrable distributions and foliations. This is the second main result in this paper. The statement is as follows.

\begin{proposition}\label{neat foliations}
Let $M$ be a manifold with boundary. The assignment of the family of maximal integral manifolds with boundary to a neatly integrable normal distribution defines a bijection between neatly integrable normal distributions and neat foliations by boundaries on $M$. Integrable normal distributions of constant rank $r$ thereby correspond to neat foliations of constant dimension $r$.
\end{proposition}

We will now prove the main results.

\section{Proofs of the main results}
\label{sec:Proof}

Our terminology for the kinds of submanifolds with boundary is in accordance with Lee's book \cite{Lee} and Ruldolph and Schmidt's book \cite{RudolphSchmidt}. In addition to this, we establish some systematic notation for use in the proof.

Let $A,B,M,N$ be manifolds with boundary and $f : M \to N$ be a map between manifolds with boundary. 
\begin{enumerate}
  \item If $A \subset M$, the restriction is denoted $f|_A : A \to N$ as usual.
  \item If $f(M) \subset B$, we will write the induced map as $f|^B : M \to B$. We call $f|^B$ the \textbf{\emph{co-restriction of $\bm{f}$ to $\bm{B}$}} if $B \subset N$ and call it the \textbf{\emph{co-extension of $\bm{f}$ to $\bm{B}$}} if $B \supset Y$. If $f(A) \subset B$, we define $f|_A^B = (f|_A)^B$ for succinctness.
  \item If $A$ is a submanifold with boundary in $M$ and $f$ is an injective immersion and $B$ is the unique submanifold with boundary in $N$ such that $f|_A^B$ is a diffeomorphism, we denote $f(A) = B$.
  \item We will denote by $\text{Pts}(A)$ the set of points of $A$ when it is necessary for distinction.
\end{enumerate}

The main step in establishing the results of this paper is to use the Stefan-Sussmann theorem in \cite{Part1} to extend the neatly integrable normal distribution into an integrable distribution on a neighbourhood of the (boundaryless) double (see definition \ref{collar extension}). We then apply the established theory without boundary to this distribution. We will now construct this distribution.

\subsection{Extension into the double}

In view the Stefan-Sussmann theorem in \cite{Part1}, we formulate the extension for collared normal distributions.

\begin{definition}[Collar extension]\label{collar extension}
Let $M$ be a manifold with (nonempty) boundary and $D$ be a normal distribution on $M$ with a collar $\Sigma : \partial M \times [0,1) \to M$.

Set $M \sqcup M = \{-1,1\} \times M$ and write $\iota_{-1},\iota_{1} : M \to M \sqcup M$ for the natural embeddings. Denote by $\dble$ the boundaryless double of $M$ where $\Pi : M \sqcup M \to \dble$ is the quotient map (see \cite[Theorem 9.29 and Example 9.32]{Lee}). The map $\Sigma_D : \partial M \times (-1,1) \to \dble$ given by $\Sigma_D(p,t) = \Pi(\iota_{\text{sgn}(t)}(\Sigma(p,|t|)))$ is an embedding. The map $\iota_D = \Pi \circ \iota_1$ is also an embedding.

Let $V_D = \Sigma_D(\partial M \times (-1,1))$ (an open submanifold of $\dble$ containing $\iota_D(\partial M)$ in $\dble$) and set $\sigma_D = \Sigma_D|^{V_D}$ (a diffeomorphism). We then obtain the open submanifold $\tilde{M} = \iota_D(\interior M) \cup V_D$ of $\dble$ containing $\iota_D(M)$. We hence obtain the embeddings $\tilde{\Sigma} = \Sigma_D|^{\tilde{M}}$ and $\tilde{\iota} = \iota_D|^{\tilde{M}}$. Define the following subsets of $T\tilde{M}$.
\begin{equation}\label{doubledist}
\begin{split}
\tilde{D}_1 &\coloneqq T\tilde{\iota}(D)\\
\tilde{D}_2 &\coloneqq \cup_{(p,t) \in \partial M \times (-1,1)}T\tilde{\Sigma}(\cdot,t)|_p((\jmath^*D)_p) + T\tilde{\Sigma}(p,\cdot )|_t(T_t(-1,1))\\
\tilde{D} &\coloneqq \tilde{D}_1 \cup \tilde{D}_2. 
\end{split}
\end{equation}
We call $\tilde{D}$ the \textbf{\emph{collar extension of $\bm{D}$}}.
\end{definition}

We will now develop the fundamental properties of the collar extension which give it its name. To this end, let $M$ be manifold with (nonempty) boundary and $D$ be a normal distribution on $M$ with a collar $\Sigma : \partial M \times [0,1) \to M$ and $\tilde{D}$ be the collar extension.

\begin{lemma}\label{doubleextend}
For all $(p,t) \in \partial M \times (-1,1)$, $\tilde{D}_{\tilde{\Sigma}(p,t)} = (\tilde{D}_2)_{\tilde{\Sigma}(p,t)}$ and, if $t \geq 0$, $(\tilde{D}_1)_{\tilde{\Sigma}(p,t)} = (\tilde{D}_2)_{\tilde{\Sigma}(p,t)}$. For $q \in \tilde{\iota}(M)$, $\tilde{D}_q = (\tilde{D}_1)_q$. 
\end{lemma}

\begin{proof}
Set $\iota : \partial M \times [0,1) \subset \partial M \times (-1,1)$. We have the relation $\tilde{\Sigma} \circ \iota = \tilde{\iota} \circ \Sigma$. Hence, for $(p,t) \in \partial M \times [0,1)$, we have,
\begin{equation*}
\begin{split}
(\tilde{D}_1)_{\tilde{\Sigma}(p,t)} &= T\tilde{\iota}(D_{\Sigma(p,t)})\\
&= T\tilde{\iota}(T\Sigma(\cdot,t)|_p((\jmath^*D)_p) + T\Sigma(p,\cdot )|_t(T_t[0,1)))\\
&= T(\tilde{\iota} \circ \Sigma)(\cdot,t)|_p((\jmath^*D)_p) + T(\tilde{\iota} \circ \Sigma)(p,\cdot )|_t(T_t[0,1))\\
&= T(\tilde{\Sigma} \circ \iota)(\cdot,t)|_p((\jmath^*D)_p) + T(\tilde{\Sigma} \circ \iota)(p,\cdot )|_t(T_t[0,1))\\
&= T\tilde{\Sigma}(\cdot,t)|_p((\jmath^*D)_p) + T\tilde{\Sigma}(p,\cdot )|_t(T_t(-1,1))\\
&= (\tilde{D}_2)_{\tilde{\Sigma}(p,t)}.
\end{split}
\end{equation*}

With this, for $(p,t) \in \partial M \times (-1,1)$, if $t \geq 0$, then we have from the above that,
\begin{equation*}
\tilde{D}_{\tilde{\Sigma}(p,t)} = (\tilde{D}_1)_{\tilde{\Sigma}(p,t)} \cup (\tilde{D}_2)_{\tilde{\Sigma}(p,t)} = (\tilde{D}_2)_{\tilde{\Sigma}(p,t)} \cup (\tilde{D}_2)_{\tilde{\Sigma}(p,t)} = (\tilde{D}_2)_{\tilde{\Sigma}(p,t)}.
\end{equation*}

If instead $t < 0$, then $\tilde{\Sigma}(p,t) \notin \tilde{\iota}(M)$ so that $(\tilde{D}_1)_{\tilde{\Sigma}(p,t)} = \emptyset$ and hence,
\begin{equation*}
\tilde{D}_{\tilde{\Sigma}(p,t)} = (\tilde{D}_1)_{\tilde{\Sigma}(p,t)} \cup (\tilde{D}_2)_{\tilde{\Sigma}(p,t)} = \emptyset \cup (\tilde{D}_2)_{\tilde{\Sigma}(p,t)} = (\tilde{D}_2)_{\tilde{\Sigma}(p,t)}.
\end{equation*}

Now, for $q \in \tilde{\iota}(M)$, if $q \in V_D$, then $q = \Sigma(p,t)$ for some $(p,t) \in \partial M \times (-1,1)$ with $t \geq 0$. Hence, from the above, we get that $\tilde{D}_q = (\tilde{D}_2)_q = (\tilde{D}_1)_q$. If instead $q \notin V_D$, then $(\tilde{D}_2)_q = \emptyset$ so that,
\begin{equation*}
\tilde{D}_{q} = (\tilde{D}_1)_{q} \cup (\tilde{D}_2)_{q} = (\tilde{D}_1)_{q} \cup \emptyset = (\tilde{D}_1)_{q}.
\end{equation*}
\end{proof}

\begin{remark}\label{distislocalproperty}
As mentioned in \cite{Part1}, we have the following. Let $M$ be a manifold with boundary and let $D \subset TM$ be a subset. If for all $p \in M$, there exists an open neighbourhood $U$ of $p$ for which $\iota^*D$ is a distribution on $U$ ($\iota : U \subset M$), then $D$ is a distribution on $M$.
\end{remark}

\begin{lemma}
The subset $\tilde{D}$ is a distribution on $\tilde{M}$.
\end{lemma}

\begin{proof}
From Lemma \ref{doubleextend}, we have that $\tilde{D}_q = (\tilde{D}_1)_q$ for all $q \in \tilde{\iota}(\interior M)$ which shows that $\tilde{D}$ is a distribution when pulled back to $\tilde{\iota}(\interior M)$. Likewise, from Lemma \ref{doubleextend}, $(\tilde{D})_q = (\tilde{D}_2)_q$ for all $q \in V_D$ which shows that $\tilde{D}$ is a distribution when pulled back to $V_D$. Hence, from Remark \ref{distislocalproperty}$, \tilde{D}$ is a distribution.
\end{proof}

\subsection{Preliminary observations}

To establish these observations, we will need the vector on the boundary trichotomy from \cite[Proposition 5.41]{Lee}. 

Let $M$ be a manifold with boundary. Let $v \in TM$ be a vector on the boundary $\partial M$. If $v$ is transverse to $\partial M$, following Lee \cite{Lee}, $v$ is said to be:
\begin{enumerate}
    \item \textbf{\emph{inward-pointing}} if there exists a curve $\gamma : [0,\epsilon) \to M$ for some $\epsilon > 0$ such that $\gamma'(0) = v$;
    \item \textbf{\emph{outward-pointing}} if there exists a curve $\gamma : (-\epsilon,0] \to M$ for some $\epsilon > 0$ such that $\gamma'(0) = v$.
\end{enumerate}

The following is Proposition 5.41 from Lee's book \cite{Lee}.

\begin{proposition}\label{trichotomy}
For $p \in \partial M$, $T_pM$ is partitioned by the set of tangent vectors, the set of inward-pointing vectors and the set of outward-pointing vectors.
\end{proposition}

Our first observation is the compatibility between the notions of ``maximal". This is established by Proposition \ref{trichotomy} and Proposition \ref{weakly embedded compatibility} in the appendix.

\begin{proposition}\label{larger}
Let $D$ be an integrable distribution on a manifold $M$. Let $L$ be a maximal integral manifold of $D$. Then, $L$ is a maximal integral manifold with boundary of $D$.
\end{proposition}

\begin{proof}
We will systematically use Theorem \ref{boundarylessmaximal} on the existence of maximal integral manifolds (of $D$).

First, let $K$ be an integral manifold with boundary. Let $p \in \partial K$. Then, take a maximal integral manifold $I$ with $p \in I$. Then, we claim that $\interior K$ and $I$ have a common point.

Indeed, consider an inward-pointing vector $u \in T_pK$. Then, $v = Tl|_p(u) \in D_p$ (where $l : K \subset M$). Take a $D$-section $X$ with $X|_p = v$. Then, take the $l$-related vector field $Y$ on $K$ and $i$-related vector field $Z$ on $I$ (where $i : I \subset M$).

Then, take integral curves $c : [0,\epsilon) \to K$ of $Y$ and $d : (-\epsilon,\epsilon) \to I$ of $Z$ both starting at $p$ (for some $\epsilon>0$). The curves $c|^M$ and $d|^M$ are both integral curves of $X$ on $M$ with $c|^M(0)=d|^M(0)$. Thus, $c|^M = d|^M$ on $[0,\epsilon)$. In particular, since $u$ is inward-pointing, we cannot have $c(t) \in \partial K$ for all $t \in [0,\epsilon)$. That is, there exists a $t \in [0,\epsilon)$ with $c(t) \in \interior K$. With such a $t \in [0,\epsilon)$, we get that,
\begin{equation*}
\interior K \ni c(t) = c|^M(t) = d^M|(t) = d(t) \in I
\end{equation*}
So, $\interior K$ and $I$ have a common point, as claimed. 

Hence, we have by maximality that $\interior K \subset I$ and hence (since $p \in I$) $\interior K \cup \{p\} \subset I$. Hence, for all $p \in \partial K$, there exists a maximal integral manifold $I$ with $\interior K \cup \{p\} \subset I$. However, by maximality, this implies there exists a maximal integral manifold $I$ such that $K \subset I$.

With this, let $K$ be an integral manifold with boundary with a point $p$ in common with $L$. Then, by the above, we may take a maximal integral manifold $I$ with $K \subset I$. Then, since $L$ and $I$ have a common point, by maximality, $I = L$ so that $K \subset L$. Now, from Proposition \ref{weakly embedded compatibility}, since $L$ is weakly embedded in $M$, $L$ is a weakly embedded manifold with boundary in $M$. From this, since $K \subset L$, and $L$ is a weakly embedded submanifold with boundary in $M$, $K$ is immersed in $L$.
\end{proof}

Our two remaining observations are from \cite{Part1}. The first is \cite[Proposition 5.4]{Part1}.

\begin{proposition}\label{interior}
Let $D$ be a normal distribution on a manifold with boundary $M$. If $K$ is an integral manifold with boundary of $D$, then $\interior K \cap \partial M = \emptyset$.
\end{proposition}

The second observation is the Stefan-Sussmann theorem for normal distributions \cite[Theorem 4.1]{Part1}.

\begin{theorem}\label{Normal Stefan-Sussmann}
Let $M$ be a manifold with boundary $\partial M \neq \emptyset$ and $D$ be a normal distribution of rank $r$ on $M$. Denote by $\imath : \interior M \subset M$ and $\jmath : \partial M \subset M$ the inclusions of the interior and the boundary into $M$. Then, the pullback $\imath^*D$ is a distribution of rank $r|_{\interior M}$ on $\interior M$ and $\jmath^*D$ is a distribution of rank $r|_{\partial M}-1$ on $\partial M$. Moreover, the following are equivalent.
\begin{enumerate}
    \item[S1.] $D$ is neatly integrable by boundaries.
    \item[S2.] $\imath^*D$ is integrable and $D$ is precollared.
    \item[S3.] $\imath^*D$ and $\jmath^*D$ are integrable and $D$ is collared.
    \item[S4.] For every $p \in M$, there exists a chart neatly adapted to $D$ at $p$.
\end{enumerate}
\end{theorem}

Recall from the prior paper \cite{Part1} that for a smooth map $F : M \to N$ between manifolds with boundary and a distribution $D$ on $N$, the subset $F^*D \coloneqq TF^{-1}(D)$ of $TM$ is called \textbf{\emph{the pullback of $\bm{D}$ on to $\bm{M}$}}. For the interested reader, the remaining definitions used in Theorem \ref{Normal Stefan-Sussmann} are contained in \cite{Part1}.

\subsection{Application of the extension}

For the following, let $M$ be a manifold with (nonempty) boundary and $D$ be a collared distribution with collar $\Sigma : \partial M \times [0,1) \to M$. We will use the notation of the collar extension from definition \ref{collar extension}. Assume further that both $\imath^*D$ and $\jmath^*D$ are integrable distributions.

\begin{lemma}\label{integralsforthedouble}
If $L$ is an integral manifold with boundary of $D$, then $\tilde{\iota}(L)$ is an integral manifold with boundary of $\tilde{D}$. If $J$ is an integral manifold of $\jmath^*D$, then $\tilde{\Sigma}(J \times (-1,1))$ is an integral manifold of $\tilde{D}$.
\end{lemma}

\begin{proof}
Assume $L$ is an integral manifold with boundary of $D$. Then, $\tilde{\iota}(L)$ is a connected immersed submanifold with boundary. Set $\tilde{l} : \tilde{\iota}(L) \subset \tilde{M}$ and $l : L \subset M$. Then, for $q \in \tilde{\iota}(L)$, writing $q = \tilde{\iota}(p)$ for some $q \in L$, we have, 
\begin{equation*}
T\tilde{l}|_q(T_q\tilde{\iota}(L)) = T\tilde{\iota}|_p(Tl|_p(T_pL)) = T\tilde{\iota}|_p(D_p) = (T\tilde{\iota}(D))_q = (\tilde{D}_1)_q.
\end{equation*}
From Lemma \ref{doubleextend}, since $q \in \tilde{\iota}(M)$, $\tilde{D}_{q} = (\tilde{D}_1)_{q}$ and hence $T\tilde{l}|_q(T_q\tilde{\iota}(L)) = \tilde{D}_{q}$. In summary, $\tilde{\iota}(L)$ is an integral manifold with boundary of $\tilde{D}$.

Assume $J$ is an integral manifold of $\jmath^*D$. Then $\tilde{\Sigma}(J \times (-1,1))$ is a connected immersed submanifold of $\tilde{M}$. Setting $\tilde{l} : \tilde{\Sigma}(J \times (-1,1)) \subset \tilde{M}$ and $k : J \times (-1,1) \subset \partial M \times (-1,1)$, we have, for $q \in \Sigma(J \times (-1,1))$, writing $q = \Sigma(p,t)$ for some $(p,t) \in \partial M \times (-1,1)$ that,
\begin{equation*}
T\tilde{l}|_q(T_q\Sigma(J \times (-1,1))) = T\Sigma|_{(p,t)}(Tk|_{(p,t)}(T_{(p,t)}(J \times (-1,1)))).
\end{equation*}
We now argue as in \cite[Lemma 5.7]{Part1}. Consider the projections $\pi_1,\pi_2$ of $\partial M \times (-1,1)$ (onto the indicated factors) and likewise $\tau_1,\tau_2$ on $J \times (-1,1)$. With $j : J \subset \partial M$, we obtain the following linear isomorphisms and relations.
\begin{align*}
\left(T\pi_1|_{(p,t)}, T\pi_2|_{(p,t)}\right) &: T_{(p,t)}(\partial M \times (-1,1)) \to T_p\partial M \times T_t(-1,1),\\
\left(T\tau_1|_{(p,t)}, T\tau_2|_{(p,t)}\right) &: T_{(p,t)}(J \times (-1,1)) \to T_p\partial L \times T_t(-1,1),\\
\pi_1 \circ k &= j \circ \tau_1,\\
\pi_2 \circ k &= \tau_2.
\end{align*}
With this, we get, 
\begin{align*}
T\pi_1|_{(p,t)}(Tk|_{(p,t)}(T_{(p,t)}(J \times (-1,1)))) &= Tj|_p(T_pJ) = (\jmath^*D)_p,\\
T\pi_2|_{(p,t)}(Tk|_{(p,t)}(T_{(p,t)}(J \times (-1,1)))) &= T_t(-1,1).
\end{align*}
Using again the fact that $\left(T\pi_1|_{(p,t)}, T\pi_2|_{(p,t)}\right)$ is a linear isomorphism, we see that for any subset  $A \subset T_{(p,t)}(\partial M \times [0,1))$,
\begin{equation*}
A = T\text{Id}(\cdot,t)|_p(T\pi_1|_{(p,t)}(A)) + T\text{Id}(p,\cdot)|_t(T\pi_2|_{(p,t)}(A)).
\end{equation*}
Hence, since $q = \tilde{\sigma}(p,t) \in V_D$, by Lemma \ref{doubleextend},
\begin{equation*}
\begin{split}
T\tilde{l}|_q(T_q\Sigma(J \times (-1,1))) &= T\Sigma|_{(p,t)}(Tk|_{(p,t)}(T_{(p,t)}(J \times (-1,1))))\\
&= T\Sigma|_{(p,t)}(T\text{Id}(\cdot,t)|_p((\jmath^*D)_p) + T\text{Id}(p,\cdot)|_t(T_t(-1,1)))\\
&= T\Sigma(\cdot,t)|_p((\jmath^*D)_p) + T\Sigma(p,\cdot)|_t(T_t(-1,1))\\
&= (D_2)_{\tilde{\Sigma}(p,t)}\\
&= \tilde{D}_{q}.
\end{split}
\end{equation*}
In summary, $\Sigma(J \times (-1,1))$ is an integral manifold of $\tilde{D}$.
\end{proof}

\begin{lemma}
The distribution $\tilde{D}$ is integrable. 
\end{lemma}

\begin{proof}
Let $x \in \tilde{M}$. If $x \in \tilde{\iota}(\interior M)$, then $x = \tilde{\iota}(p)$ for some $p \in \interior M$. Then, taking an integral manifold $I$ of $\imath^*D$ with $p \in I$, from Lemma \ref{integralsforthedouble} we have that $\tilde{\iota}(I)$ is an integral manifold of $\tilde{D}$ where $x \in \tilde{\iota}(I)$. 

Otherwise, $x \in V_D$ so that $x = \tilde{\Sigma}(p,t)$ for some $(p,t) \in \partial M \times (-1,1)$. Then, taking an integral manifold $J$ of $\jmath^*D$, from Lemma \ref{integralsforthedouble} we have that $\tilde{\Sigma}(J \times (-1,1))$ is an integral manifold of $\tilde{D}$ where $x \in \tilde{\Sigma}(J \times (-1,1))$.
\end{proof}

\begin{lemma}\label{final work}
If $\tilde{L}$ is an integral manifold of $\tilde{D}$, then the subset $\tilde{\iota}^{-1}(\tilde{L})$ of $M$ can be given the structure of manifold with boundary $\tilde{\iota}^*\tilde{L}$ whereby $\tilde{\iota}|_{\tilde{\iota}^*\tilde{L}}^{\tilde{L}}$ is an embedding. Moreover, $\tilde{\iota}^*\tilde{L}$ is an weakly embedded submanifold with boundary of $M$ which satisfies $\partial( \tilde{\iota}^*\tilde{L} )\subset \partial M$ and $T_pl(T_p\tilde{\iota}^*\tilde{L}) = D_p$ for $p \in \tilde{\iota}^*\tilde{L}$, where $l : \tilde{\iota}^*\tilde{L} \subset M$.
\end{lemma}

\begin{proof}
First, we have that $\tilde{\iota}(\interior M) \cap \tilde{L}$ and $V_D \cap \tilde{L}$ are open submanifolds of $\tilde{L}$. Now, we have that $f = (\pi_2 \circ \sigma_D^{-1})|_{| V_D \cap \tilde{L}}^{\mathbb{R}}$ is regular on the immersed submanifold $V_D \cap \tilde{L}$ of $V_D$. Hence, the sublevel set $f^{-1}((-\infty,0])$ is a regular domain in $V_D \cap \tilde{L}$ (see Proposition 5.47 in Lee's book \cite{Lee}) whereby $\partial f^{-1}((-\infty,0]) = f^{-1}(0) = \tilde{\partial M} \cap (V_D \cap \tilde{L}) = \tilde{\partial M} \cap \tilde{L}$. That is, $\tilde{\iota}(V)\cap\tilde{L}$ is a regular domain in $\tilde{L}$ where $\partial (\tilde{\iota}(V)\cap\tilde{L}) = \tilde{\partial M} \cap \tilde{L}$.

With this, we get that $(\tilde{\iota}(\interior M) \cap \tilde{L}) \cup (\tilde{\iota}(V)\cap\tilde{L}) = \tilde{\iota}(M) \cap \tilde{L}$ is an embedded submanifold with boundary of $\tilde{L}$. Hence, we can endow $\tilde{\iota}^{-1}(\tilde{L})$ with the unique smooth structure $\tilde{\iota}^*\tilde{L}$ so that the map $\tilde{h} = \tilde{\iota}|_{\tilde{\iota}^*\tilde{L}}^{\tilde{\iota}(M) \cap \tilde{L}}$ is a diffeomorphism. It is hence clear from this that $\tilde{\iota}^*\tilde{L}$ is a submanifold with boundary of $M$ whereby $\tilde{\iota}|_{\tilde{\iota}^*\tilde{L}}^{\tilde{L}}$ is an embedding (so that $\partial (\tilde{\iota}^*\tilde{L}) \subset \partial M$). We now show that $\tilde{\iota}^*\tilde{L}$ is weakly embedded.

Let $f : S \to M$ be a smooth map between manifolds with boundary with $f(S) \subset \tilde{\iota}^*\tilde{L}$. Then, $\tilde{\iota} \circ f : S \to \tilde{M}$ is smooth. Then, notice that $(\tilde{\iota} \circ f)(S) \subset \tilde{L}$. From this, since $\tilde{L}$ is weakly embedded, Proposition \ref{weakly embedded compatibility} gives us that $(\tilde{\iota} \circ f)|^{\tilde{L}}$ is smooth. Moreover, we can write $(\tilde{\iota} \circ f)|^{\tilde{L}} = \tilde{\iota}|_{\tilde{\iota}^*\tilde{L}}^{\tilde{L}} \circ f|^{\tilde{\iota}^*\tilde{L}}$. Hence, because $\tilde{\iota}|_{\tilde{\iota}^*\tilde{L}}^{\tilde{L}}$ is an embedding, we have that $f|^{\tilde{\iota}^*\tilde{L}}$ is smooth. Hence, $\tilde{\iota}^*\tilde{L}$ is weakly embedded, as claimed.

With this, writing $\tilde{l} : \tilde{\iota}(M) \cap \tilde{L} \subset \tilde{M}$, we have that $\tilde{l} \circ \tilde{h} = \tilde{\iota} \circ l$. So, for $p \in \tilde{\iota}^*\tilde{L}$ we have from Lemma \ref{doubleextend} that $\tilde{D}_{\tilde{\iota}(p)} = (\tilde{D}_1)_{\tilde{\iota}(p)} = T\tilde{\iota}|_p(D_p)$ and hence,
\begin{equation*}
\begin{split}
T\tilde{\iota}|_p(Tl|_p(T_p\tilde{\iota}^*\tilde{L})) &= T\tilde{\iota} \circ l|_p(T_p\tilde{\iota}^*\tilde{L})\\
&= T\tilde{l} \circ \tilde{h}|_p(T_p\tilde{\iota}^*\tilde{L})\\
&= T\tilde{l}|_{\tilde{h}(p)}(T\tilde{h}|_p(T_p\tilde{\iota}^*\tilde{L}))\\
&= T\tilde{l}|_{\tilde{h}(p)}(T_{\tilde{h}(p)}\tilde{L})\\
&= T\tilde{l}|_{\tilde{\iota}(p)}(T_{\tilde{\iota}(p)}\tilde{L})\\
&= \tilde{D}_{\tilde{\iota}(p)}\\
&= T\tilde{\iota}|_p(D_p).
\end{split}
\end{equation*}

So that, since $\tilde{\iota}$ is an embedding, $Tl|_p(T_p\tilde{\iota}^*\tilde{L}) = D_p$.
\end{proof}

\begin{lemma}\label{existsmax}
At every point $p \in M$, there exists a neat integral manifold $L$ of $D$ containing $p$ such that $L$ is weakly embedded in $M$ and maximal.
\end{lemma}

\begin{proof}
Let $p \in M$. Then, $\tilde{\iota}(p) \in \tilde{M}$. Since $\tilde{D}$ is integrable, from Theorem \ref{boundarylessmaximal}, there exists a maximal integral manifold $\tilde{L}$ of $\tilde{D}$ containing $\tilde{p}$. Now, let $L$ denote the connected component of $\tilde{\iota}^*\tilde{L}$ containing $p$ as in Lemma \ref{final work}. Then, we see from this Lemma that $L$ is an integral manifold with boundary $\partial L \subset \partial M$. So that $L$ is neat. 

We will now show that $L$ is weakly embedded. Indeed, since $L$ is an open submanifold with boundary of $\tilde{\iota}^*\tilde{L}$, from Theorem \ref{embedded with bdry implies weak with boundary}, we have that $L$ is a weakly embedded submanifold of $\tilde{\iota}^*\tilde{L}$ and hence $L$ is a weakly embedded submanifold of $M$.

We will how show that $L$ is maximal. Indeed, let $K$ be an integral manifold with boundary of $D$ with a point $q$ in common with $L$. Then, from Lemma \ref{integralsforthedouble}, we have that $\tilde{\iota}(K)$ is an integral manifold with boundary of $\tilde{D}$ and moreover, $\tilde{\iota}(q)$ is a point common to $\tilde{L}$. Hence, by Proposition \ref{larger}, $\tilde{\iota}(K) \subset \tilde{L}$.

Hence, (by injectivity) we have $K \subset \iota^*\tilde{L}$. Then, since $K$ is immersed in $M$ and $\iota^*\tilde{L}$ is weakly embedded in $M$, $K$ is an immersed submanifold with boundary of $\iota^*\tilde{L}$. Hence, since $K$ is connected, we have that $\text{Pts}(K)$ is a connected subset of $\iota^*\tilde{L}$. Hence, because $q \in K$ and $q \in L$, we must have that $K \subset L$. Then, since $L$ is weakly embedded in $M$, $K$ is an immersed submanifold of $L$.
\end{proof}

\begin{lemma}\label{neatintsareweak}
Let $L$ be a maximal integral manifold with boundary of $D$. Let $K$ be a neat integral manifold with boundary of $D$ contained by $L$. Then $K$ is an open submanifold with boundary of $L$.
\end{lemma}

\begin{proof}
Since $K \subset L$ and $L$ is weakly embedded, $K$ is an immersed submanifold with boundary of $L$. Since both $K$ and $L$ are neat, Proposition \ref{interior} gives $\partial K = \partial M \cap K$ and $\partial L = \partial M \cap L$ so that $\partial K \subset \partial L$. Since $K$ has codimension $0$ in $L$, Proposition \ref{codim 0 bdry} gives us that $K$ is an open submanifold with boundary of $L$.
\end{proof}

\subsection{Maximality and weak embeddedness}

The proofs of our results on maximality and weak embeddedness are immediate from Lemmas \ref{existsmax} and \ref{neatintsareweak} as follows.

\begin{proof}[Proof of Theorem \ref{normal dist maximal}]
Let $M$ be a manifold with boundary and $D$ be a normal distribution on $M$ which is neatly integrable. Let $p \in M$. From Lemma \ref{existsmax}, there exists a maximal integral manifold $L$ of $D$ containing $p$. Now, if $L_1,L_2$ were maximal integral manifolds with boundary of $D$ containing $p$. Then $L_1$ and $L_2$ are immersed submanifolds with boundary of each other and hence $L_1 = L_2$. That is, at each point $p \in M$, there exists a maximal integral manifold with boundary $L$ containing $p$.

Now, let $K$ be a maximal integral manifold of $D$. Then, taking an element $p \in K$, from Lemma \ref{existsmax}, there exists a maximal integral manifold $L$ with boundary passing through $p$ which is neat. By the uniqueness above, we have that $K = L$ and hence $K$ is neat. 
\end{proof} 

\begin{proof}[Proof of Proposition \ref{normal dist weak embedd}]
Let $M$ be a manifold with boundary and $D$ be a normal distribution on $M$ which is neatly integrable. Let $K$ be an neat integral manifold with boundary of $D$. From Lemma \ref{existsmax}, there exists a neat integral manifold $L$ of $D$ containing $p$ such that $L$ is weakly embedded in $M$ and maximal. Thus, $K \subset L$. Hence, by Lemma \ref{neatintsareweak}, $K$ is an open submanifold with boundary of $L$. From this, it follows that $K$ is weakly embedded in $L$ from Theorem \ref{embedded with bdry implies weak with boundary}. Since $L$ is weakly embedded in $M$, it easily follows that $K$ is weakly embedded in $M$.
\end{proof}

\subsection{Neat foliations}

We will now prove our 1-1 correspondence with neat foliations. We will keep track of the set theory involved.

\begin{proof}[Proof of Proposition \ref{neat foliations}]
Let $M$ be a manifold with boundary. Let 
\begin{equation*}
\mathcal{D}_N \coloneqq \{D \subset TM : D \text{ is an integrable normal distribution on $M$} \}.
\end{equation*}
Then $\mathcal{D}_N$ is the set of integrable normal distributions on $M$. Now, let $D \in \mathcal{D}_N$. From Theorem \ref{normal dist maximal}, we may replace each point $p \in M$ with the unique maximal integral manifold with boundary of $D$ containing $p$ to form the unique partition $\mathcal{F}_D$ of $M$ by the maximal integral manifolds with boundary of $D$. 

We now show that $\mathcal{F}_D$ is a neat foliation by manifolds with boundary. From Proposition \ref{normal dist weak embedd}, $\mathcal{F}_D$ is, in particular, a partition of $M$ into weakly embedded submanifolds with boundary. Now, let $p \in M$. Then, by Theorem \ref{Normal Stefan-Sussmann}, there exists a chart $(U,\varphi)$ neatly adapted to $D$ at $p$. That is, setting $n = \dim M$, $r : M \to \mathbb{N}_0$ to be the rank of $D$, and denoting by $\partial_1,...,\partial_n$ the local vector fields on $M$ induced by $(U,\varphi)$, the following holds.
\begin{enumerate}
    \item $\varphi(p) = 0$ and $\varphi(U) = (-\epsilon,\epsilon)^n$ if $p \in \interior M$ and $\varphi(U) = (-\epsilon,\epsilon)^n \cap \mathbb{H}^n$ otherwise,
    \item $\partial_{n-r(p)+1},...,\partial_n$ are local $D$-sections,
    \item For all $c \in (-\epsilon,\epsilon)^{n-r(p)}$, denoting $U_c = \varphi^{-1}\left(\{c\} \times (-\epsilon,\epsilon)^{r(p)}\right)$ if $r(p) < n$ and $U_c = U$ otherwise, the rank $r$ of $D$ is constant along $U_c$.
\end{enumerate}
Now, we have that $R = r(p)$ is the dimension of the unique $K \in \mathcal{F}_D$ containing $p$ since $K$ is an integral manifold with boundary of $D$. Let $L \in \mathcal{F}_D$. Consider $\partial_i$ where $n-r(p)+1 \leq i \leq n$. Then, since $\partial_i$ is a local $D$-section, and $L$ is a maximal integral manifold with boundary, by a similar argument to that of Proposition 5.2 in \cite{Part1} on maximal integral manifolds contain integral curves they intersect, we get that the integral curves of $\partial_i$ intersecting $L$ are contained by $L$. Altogether, $(U,\varphi)$ satisfies,
\begin{enumerate}
  \item $\varphi(p) = 0$ and $\varphi(U) = (-\epsilon,\epsilon)^n$ for some $\epsilon > 0$ if $p \in \interior M$ and $\varphi(U) = (-\epsilon,\epsilon)^n \cap \mathbb{H}^n$ for some $\epsilon > 0$ otherwise.
  \item For all $L \in \mathcal{F}$, and for any coordinate vector field $\partial_i$ where $n-R+1 \leq i \leq n$, all integral curves of $\partial_i$ intersecting $L$ are contained in $L$.
\end{enumerate}
Hence, $\mathcal{F}_D$ is a neat foliation by manifolds with boundary.

Let $\mathcal{F}$ be a neat foliation by manifolds with boundary. Define,
\begin{equation*}
D_{\mathcal{F}} \coloneqq \{v \in TM : v \in Tl(TL) \text{ for some } L \in \mathcal{F} \text{ where } l : L \subset M\}
\end{equation*}

Clearly $D_{\mathcal{F}} \subset TM$ and $(D_{\mathcal{F}})_p$ is a vector subspace of $T_pM$ for each $p \in M$. Now, let $p \in M$. Then, take a chart $(U,\varphi)$ satisfying the following (where $R$ is the dimension of the unique $K \in \mathcal{F}$ containing $p$).
\begin{enumerate}
  \item $\varphi(p) = 0$ and $\varphi(U) = (-\epsilon,\epsilon)^n$ for some $\epsilon > 0$ if $p \in \interior M$ and $\varphi(U) = (-\epsilon,\epsilon)^n \cap \mathbb{H}^n$ for some $\epsilon > 0$ otherwise.
  \item For all $L \in \mathcal{F}$, and for any coordinate vector field $\partial_i$ where $n-R+1 \leq i \leq n$, all integral curves of $\partial_i$ intersecting $L$ are contained in $L$.
\end{enumerate}
From the second property, we see that the coordinate vector fields $\partial_i$ for $n-R+1 \leq i \leq n$ take values in $D_{\mathcal{F}}$. So, since $\dim (D_{\mathcal{F}})_p = R$ and the $\partial_i$ ($i \in \{1,...,n\}$) are linearly independent, $(D_{\mathcal{F}})_p = \text{span}\{\partial_i|_p : n-R+1\leq i \leq n\}$. Taking a smooth bump function $\rho : U \to \mathbb{R}$ with $\rho(p) = 1$, denoting by $\rho\partial_i$ the obvious vector field on $M$, we have $(D_{\mathcal{F}})_p = \text{span}\{\rho\partial_i|_p : n-R+1\leq i \leq n\}$. So that any vector $v \in (D_{\mathcal{F}})_p$ may be extended by a $D_{\mathcal{F}}$-section. Moreover, if $p \in \partial M$, then the vector $\partial_n|_p$ is transverse to the boundary. Moreover, $\partial_n|_p \in (D_{\mathcal{F}})_p$ so that $(D_{\mathcal{F}})_p$ contains a vector transverse to the boundary. Hence, $D_{\mathcal{F}}$ is a normal distribution. We also see by construction that $\mathcal{D}_F$ is integrable. Altogether, $D_{\mathcal{F}} \in \mathcal{D}_N$.

Now, setting $D = D_{\mathcal{F}}$, we have $D \in \mathcal{D}_N$ so that there exists the unique partition $\mathcal{F}_D$ of $D$ by the maximal integral manifolds with boundary of $M$. Let $L \in \mathcal{F}_D$. Then, take a point $p \in L$ and consider the unique $K \in \mathcal{F}$ with $p \in K$. Since $K$ is a neat integral manifold with boundary of $D$, we have that $K \subset L$. Moreover, from Lemma \ref{neatintsareweak}, $K$ is an open submanifold of $L$. Hence, we get the partition $\{K \in \mathcal{F} : K \cap L \neq \emptyset\}$ of $L$ into open submanifolds with boundary of $L$. Since $L$ is connected, $\{K \in \mathcal{F} : K \cap L \neq \emptyset\}$ must be a singleton. Thus, $L \in \mathcal{F}$. Hence, $\mathcal{F}_D \subset \mathcal{F}$. So, since both $\mathcal{F}_D$ and $\mathcal{F}$ are partitions, we get $\mathcal{F}_D = \mathcal{F}$.

From this, we see that the set $\mathcal{F}_N$ of neat foliations by manifolds with boundary exists and that we have a bijection $\mathcal{D}_N \to \mathcal{F}_N$ given by $D \mapsto \mathcal{F}_D$ where $\mathcal{F}_D$ is the unique partition of $M$ by the maximal integral manifolds with boundary of $D$.
\end{proof}

\section{Remarks on integral manifold structure}
\label{sec:remark}

The reason we cannot ask for anything better than ``immersed" in the definition of maximal integral manifold with boundary is because of the presence of boundaries. Boundaries allow for integral manifolds with artificially low regularity as the following example illustrates. 

\begin{example}\label{necessaryweaken}
Polar coordinates $P : \mathbb{R}^+ \times [0,2\pi) \to \mathbb{R}^2 \backslash \{0\}$ given by $P(r,\theta) = (r\cos{\theta}, r \sin{ \theta})$ can be used, for any $\delta > 0$, to give a connected immersed submanifold $P((0,\delta)\times [0,2\pi))$ with (non-empty) boundary of the punctured disk of radius $\delta$, $D_{\delta} \backslash \{0\}$, such that $\text{Pts}(P((0,\delta)\times [0,2\pi))) = \text{Pts}(D_{\delta} \backslash \{0\})$. Let $D$ be a distribution on a manifold with boundary $M$ with an integral manifold with boundary $L$ of dimension at least $2$ which is weakly embedded in $M$. The above construction implies that there is an integral manifold with boundary $L^*$ of $D$ with $L^* \subset L$ but $L^*$ is not a weakly embedded in submanifold with boundary of $L$. This is a lot stronger than $L^*$ not being an open submanifold.
\end{example}

On the other hand, the weakly embedded submanifolds are the unique minimisers of their boundary points in the following sense.

\begin{proposition}\label{weakly embedded minimises boundary}
Let $L$ be a weakly embedded submanifold with boundary of a manifold with boundary $M$. Let $\tilde{L}$ be an immersed submanifold with boundary of $M$ such that $\text{Pts}(\tilde{L}) = \text{Pts}(L)$. Then, $\text{Pts}(\partial L) \subset \text{Pts}(\partial \tilde{L})$. If in addition $\text{Pts}(\partial L) = \text{Pts}(\partial \tilde{L})$, then $\tilde{L} = L$.
\end{proposition}

\begin{proof}
By weak embeddedness, we get that $i : \tilde{L} \to L$ is a bijective immersion. Hence, $\dim{\tilde{L}} \leq \dim{L}$. If $\dim{L} = 0$, everything is trivial. So, assume $\dim{L} > 0$. Then, since $\text{Pts}(\tilde{L}) = \text{Pts}(L)$, we have that $\text{Pts}(\tilde{L})$ does not have measure zero in $L$. Hence by a corollary of Sard's Theorem (see \cite[Corollary 6.12]{Lee}), we cannot have $\dim{\tilde{L}} < \dim{L}$. So, $\dim{\tilde{L}} = \dim{L}$. Hence, $i : \tilde{L} \to L$ is a submersion. From Proposition \ref{interiorinc}, we must have that $\text{Pts}(\interior \tilde{L}) \subset \text{Pts}(\interior L)$. Hence, $\text{Pts}(\partial L) \subset \text{Pts}(\partial \tilde{L})$.

Next, assuming $\text{Pts}(\partial L) = \text{Pts}(\partial \tilde{L})$, we have both $\text{Pts}(\interior \tilde{L}) \subset \text{Pts}(\interior L)$ and $\text{Pts}(\partial \tilde{L}) \subset \text{Pts}(\partial L)$ whereby $\tilde{L}$ is a codimension $0$ immersed submanifold of $L$. Then, by Proposition \ref{codim 0 bdry}, $\tilde{L}$ is an open submanifold with boundary of $L$ and hence from $\text{Pts}(\tilde{L}) = \text{Pts}(L)$, we get that $\tilde{L} = L$.
\end{proof}

This is a generalisation of the following uniqueness result \cite[Theorem 5.33]{Lee} in the boundaryless case.

\begin{theorem}\label{uniqueweak}
If $M$ is a smooth manifold and $S$ is a weakly embedded submanifold, then $S$ has only one topology and smooth structure with respect to which it is an immersed submanifold.
\end{theorem}

In fact, we know from Proposition \ref{interior} that neat integral manifolds $L$ with boundary of a normal distribution on a manifold $M$ with boundary satisfy $\partial L = \partial M \cap L$ and from Proposition \ref{normal dist weak embedd} that $L$ is weakly embedded in $M$. With respect to the immersed submanifold and weakly embedded submanifolds with boundary satisfying this condition, Proposition \ref{weakly embedded minimises boundary} gives a formally identical generalisation of Theorem \ref{uniqueweak} as follows.

\begin{corollary}
If $M$ is a smooth manifold with boundary and $S$ is a weakly embedded submanifold with boundary satisfying $\partial S = \partial M \cap S$, then $S$ has only one topology and smooth structure with respect to which it is an immersed submanifold with boundary satisfying $\partial S = \partial M \cap S$.
\end{corollary}

\section{Competing interests}

The authors declare no competing interests.

\section{Acknowledgements}

This paper was written while the first author received an Australian Government Research Training Program Scholarship at The University of Western Australia.

\appendix

\section{Submanifolds with boundary}
\label{sec:appendix}

Here we will discuss the results about weakly embedded submanifolds used in the proofs of the main results (Section \ref{sec:Proof}). Firstly, used in Proposition \ref{larger} and Lemma \ref{final work}, was the fact that the two notions of ``weakly embedded" are compatible as follows.

\begin{proposition}\label{weakly embedded compatibility}
If $L$ is a manifold which is weakly embedded in a manifold $M$, then the manifold with boundary $L$ is weakly embedded in the manifold with boundary $M$. 
\end{proposition}

To prove it, we slightly generalize an argument given in Proposition 1.7.2 of Rudolph and Schmidt's book \cite{RudolphSchmidt}. For this we require their definitions in our context.

Let $M$ be a smooth manifold. Following Rudolph and Schmidt \cite{RudolphSchmidt}, a continuous map $\gamma : [a,b] \to M$ (where $a < b$) is called a \textbf{\emph{piecewise smooth curve in $\bm{M}$}} if there exists $t_0 = a<t_1,...,t_r<b = t_{r+1}$ such that $\gamma|_{(t_i,t_{i+1})}$ is smooth for all $i=0,...,r$. A subset $A$ of $M$ is said to be \textbf{\emph{$\bm{C^{\infty}}$-arcwise connected relative to $\bm{M}$}} if any two points can be joined by a piecewise smooth curve in $M$. A subset $A$ is called a \textbf{\emph{$\bm{C^{\infty}}$-arcwise component of $\bm{N}$}} if it is an equivalence class of the equivalence relation of $C^{\infty}$-arcwise joinability of points.

\begin{example}
In contrast to piecewise smooth curves, the relation between points obtained by smooth curves $\gamma : [a,b] \to M$ is not transitive in general. The unit square in the plane is $C^{\infty}$-arcwise connected but opposing vertices are not joined by any smooth curve, despite both being joined by smooth curves to a common vertex.
\end{example}

The following is a direct result of Proposition 1.7.2 of Rudolph and Schmidt \cite{RudolphSchmidt}.

\begin{lemma}\label{Proposition 1.7.2}
Assume $L$ is a manifold which is weakly embedded in a manifold $M$. Set $l = \dim{L}$ and $n = \dim{M}$. Then, there exists a countable covering $\{(L_i,\psi_i)\}_{i \in \mathbb{N}}$ by charts of $L$ such that, for each $i$, there is a chart $(V_i,\Psi_i)$ on $M$ such that:
\begin{itemize}
  \item $L_i \subset V_i$ whereby $(\psi_i(x),0) = \Psi_i(x)$ for $x \in L_i$, where $0 \in \mathbb{R}^{n-l}$. 
  \item $L_i$ is a $C^{\infty}$-arcwise connected component of $V_i \cap L$ relative to $M$.
\end{itemize}
\end{lemma}

With this we can now establish our compatibility.

\begin{proof}[Proof of Proposition \ref{weakly embedded compatibility}]
We fill in all details for self-containment.

Let $S$ be a manifold with boundary and let $f : S \to M$ be smooth whereby $f(S) \subset L$. Then, since $L$ is weakly embedded, we have that $(f|^L)|_{\interior S} = (f|_{\interior S})|^L$ is smooth. It remains to check smoothness about boundary points. Inherit the notation from Lemma \ref{Proposition 1.7.2} and set $s = \dim{S}$.

Let $p \in \partial S$. Then, take an $i$ for which $f(p)\in L_i$. Then, since $f^{-1}(V_i)$ is open in $S$, we can take a chart $(U,\varphi)$ about $p$ with $\varphi(U) = B_r(0) \cap \mathbb{H}^s$ and $\varphi(p) = 0$ such that $f(U) \subset V_i$. Then, $f(U) \subset V_i \cap L$. Since $B_r(0) \cap \mathbb{H}^s$ is a convex subset of $\mathbb{R}^s$, it easily follows that $f(U)$ is $C^{\infty}$-arcwise connected relative to $M$. Hence, $f(U) \subset L_i$.

Consider the representative $\psi_i \circ f|^L \circ \varphi_i^{-1} : \varphi_i(U) \to \psi_i(L_i)$ of $f|^L$. By assumption, $\Psi_i \circ f \circ \varphi_i^{-1} : \varphi_i(U) \to \Psi_i(V_i)$ is a smooth representative of $f$. From the relationship $\Psi_i \circ f \circ \varphi_i^{-1}|^{\psi_i(L_i)} = \left(\psi_i \circ f|^L \circ \varphi_i^{-1},0 \right)$, we see that $\psi_i \circ f|^L \circ \varphi_i^{-1}$ is smooth. Since $p \in \partial S$ was arbitrary, $f|^L$ is smooth.
\end{proof}

Used in Lemmas \ref{existsmax} and \ref{neatintsareweak} was Theorem 5.53 in Lee's book \cite{Lee} which is stated here in terms of our notion of weakly embedded for submanifolds with boundary.

\begin{theorem}\label{embedded with bdry implies weak with boundary}
If $N$ is an embedded submanifold with boundary of a manifold with boundary $M$, then $N$ is weakly embedded in $M$.
\end{theorem}

Used in Lemma \ref{neatintsareweak} was the following proposition about codimension $0$ submanifolds as follows.

\begin{proposition}\label{codim 0 bdry}
If $M$ is an immersed submanifold of codimension $0$ of a manifold with boundary $N$, then $\interior M \subset \interior N$. If in addition $\partial M \subset \partial N$, then $M$ an open submanifold with boundary of $N$.
\end{proposition}

To be thorough, to prove Proposition \ref{codim 0 bdry}, we first extend The Inverse Function Theorem to $\mathbb{H}^n$ in an elementary way as follows.

\begin{proposition}\label{invfunc}
Let $f : U \to V$ be a smooth map between open subsets of $\mathbb{H}^n$ such that $f(\partial U) \subset \partial V$ and $f(\interior U) \subset \interior V$. Let $p \in \partial U$ be such that $Df|_p$ is invertible. Then, there exists open subsets $U'$ and $V'$ of $\mathbb{H}^n$ with $p \in U'$ such that $f' : U' \to V'$ is a diffeomorphism (where $f' = f|_{U'}^{V'}$).
\end{proposition}

\begin{proof}
We will be very thorough with details.

By definition, at $p$, we have a smooth map $H_1 : U_1 \to \mathbb{R}^n$ on a neighbourhood of $p$ such that $H_1 = f$ on $U_1 \cap U$. So, by the Inverse Function Theorem, there is a diffeomorphism $H_2 : U_2 \to V_2$ on a neighbourhood $U_2 \subset U_1$ of $p$ where $H_2 = H_1$ on $U_2$ and hence $H_2 = f$ on $U_2 \cap U$. Hence (by restriction and co-restriction), we get a diffeomorphism $F : B_r(p) \to W$ where $B_r(p) \cap \mathbb{H}^n \subset U$ and $W \cap \mathbb{H}^n \subset V$ and $F = f$ on $B_r(p) \cap \mathbb{H}^n$ (for some $r > 0$).

Consider the $n^{\text{th}}$ component $F^n : B_r(p) \to \mathbb{R}$ of $F$. Consider then the vector field $\nabla F^n$ and the embedded submanifold $S = B_r(p) \cap \partial \mathbb{H}^n$ of $B_r(p)$ of codimension $1$. Notice that $F^n|_S = 0$. It follows that since $\nabla F^n \neq 0$ everywhere, $\nabla F^n$ is nowhere tangent to $S$.

Let $\Psi : \mathcal{D} \to B_r(p)$ be the flow of $\nabla F^n$ and $\mathcal{O} = (S \times \mathbb{R}) \cap \mathcal{D}$. By the Flowout Theorem (from Lee's book \cite{Lee}), there exists a positive function $\delta : S \to \mathbb{R}$ such that $G = \Psi|_{\mathcal{O}_{\delta}}^X$ is a diffeomorphism onto an open subset $X$ of $B_r(p)$ where $\mathcal{O}_{\delta} = \{(p,t) \in \mathcal{O} : |t| < \delta(p) \}$.

Now, since $X$ is open in $B_r(p)$ and $B_r(p) \cap \partial \mathbb{H}^n = S \subset X$, there exists $x \in X$ with $x \in \interior \mathbb{H}^n$. Then, $x \in X \cap \mathbb{H}^n$ and hence $x \in B_r(p) \cap \mathbb{H}^n$ so $x \in U$ and thus (due to $x \in \interior \mathbb{H}^n$), we have $x \in \interior U$ and hence $F(x) = f(x) \in \interior V$ so that $F^n(x) > 0$. Then, since $x \in X$ and $G$ is a diffeomorphism, we have that $x = G(x_0,t^*)$ for some $x_0 \in S$ where $|t^*| < \delta(p)$. Now, the function $t \mapsto F^n(G(x_0,t))$ is (strictly) increasing. Hence, if $t^* \leq 0$, we would have that $F^n(x) = F^n(G(x_0,t^*)) \leq F^n(G(x_0,0)) = F^n(x_0) = 0$, contradicting $F^n(x) > 0$. Thus $t^* > 0$.

Now, consider $\mathcal{O}_{\delta}^+ = \{(p,t) \in \mathcal{O} : 0 < t < \delta(p)\}$. Then, $\mathcal{O}_{\delta}^+$ is connected (since $S$ is) and open in $\mathcal{O}_{\delta}$ so that $G(\mathcal{O}_{\delta}^+)$ is connected and open in $B_r(p)$. Now, since $S \cap G(\mathcal{O}_{\delta}^+) = \emptyset$, we have that $G(\mathcal{O}_{\delta}^+) \subset B_r(p) \backslash \partial \mathbb{H}^n$ and thus, since $x = G(x_0,t^*) \in B_r(p) \cap \interior \mathbb{H}^n$, we have that $G(\mathcal{O}_{\delta}^+) \cap (B_r(p) \cap \interior \mathbb{H}^n) \neq \emptyset$. Hence, we must have that $G(\mathcal{O}_{\delta}^+) \subset B_r(p) \cap \interior \mathbb{H}^n$. Hence, setting $\mathcal{O}_{\delta}^{\geq} = \{(p,t) \in \mathcal{O} : 0 \leq t < \delta(p)\}$, we have $G(\mathcal{O}_{\delta}^{\geq}) \subset B_r(0) \cap \mathbb{H}^n$.

With this, we show that $F(X \cap \mathbb{H}^n) = F(X) \cap \mathbb{H}^n$. 

Let $y \in F(X \cap \mathbb{H}^n)$. Then, $y \in F(X)$ and $y \in F(B_r(0) \cap \mathbb{H}^n)$ and so $y \in f(B_r(0) \cap \mathbb{H}^n)$ and thus $y \in \mathbb{H}^n$. Altogether, $y \in F(X) \cap \mathbb{H}^n$.

Let $y \in F(X) \cap \mathbb{H}^n$. Then, $y \in F(X)$ so $y = F(x)$ for some $x \in X$. Since $G$ is a diffeomorphism, we have $x = G(x_0,t)$ for some $|t| < \delta(p)$. Now, if $t < 0$, we would have $F^n(G(x_0,t)) < F^n(G(x_0,0)) = F^n(x_0) = 0$ so that $y = F(x) \notin \mathbb{H}^n$. By contradiction, $t \geq 0$. Hence, $G(x_0,t) \in G(\mathcal{O}_{\delta}^{\geq})$ and thus $x = G(x_0,t) \in B_r(0) \cap \mathbb{H}^n \subset \mathbb{H}^n$. Hence, $x \in X \cap \mathbb{H}^n$. Altogether, $y \in F(X \cap \mathbb{H}^n)$.

One takes $U' = X \cap \mathbb{H}^n$ (clearly $p \in U'$) and $V' = F(X \cap \mathbb{H}^n) = F(X) \cap \mathbb{H}^n$ and $f' = f|_{U'}^{V'} = F|_{U'}^{V'}$.
\end{proof}

Proposition \ref{trichotomy} establishes the following interior inclusion result. 

\begin{proposition}\label{interiorinc}
Let $f : M \to N$ be a submersion between manifolds with boundary. Then, $f(\interior M) \subset \interior N$.
\end{proposition}

\begin{proof}
Let $q \in f(\interior M)$. Then, let $v \in T_{q}N$. Then, there is a $u \in T_p M$ with $Tf|_p(u) = v$ where $p \in \interior M$. For such a $u$ and $p$, there is a curve $\gamma : (-\epsilon,\epsilon) \to M$ for some $\epsilon > 0$ such that $\gamma'(0) = u$. Hence, the curve $f \circ \gamma : (-\epsilon,\epsilon) \to N$ is with $(f \circ \gamma)'(0) = v$. Thus, by Proposition \ref{trichotomy}, $q \in \interior N$. Thus, $f(\interior M) \subset \interior N$.
\end{proof}

Hence, an extension of the global inverse function theorem is a corollary of Proposition \ref{interiorinc} and Lemma \ref{invfunc} as follows.

\begin{corollary}\label{globinvfunc}
Let $f : M \to N$ be a smooth map between manifolds with boundary such that $T_pf$ is invertible for all $p \in M$. Then $f(\interior M) \subset \interior N$. In addition, if $f(\partial M) \subset \partial N$, then $f$ is a local diffeomorphism.
\end{corollary}

This easily establishes our required codimension $0$ result.

\begin{proof}[Proof of Proposition \ref{codim 0 bdry}]
We have that the inclusion $i : M \to N$ is smooth and $T_pi$ is invertible for all $p \in M$. Then, apply Corollary \ref{globinvfunc}.
\end{proof}

\bibliography{Neat_foliations_by_manifolds_with_boundary}

\end{document}